\documentclass{article}
\usepackage{amssymb}
\usepackage{amsthm}
\usepackage{amsmath}

\newtheorem{theorem}{Theorem}[section]
\newtheorem{prop}[theorem]{Proposition}
\newtheorem{conjecture}[theorem]{Conjecture}
\newtheorem{lemma}[theorem]{Lemma}

\theoremstyle{definition}
\newtheorem{definition}{Definition}

\begin{document}

\title{Sums of Cubes in Quaternion Rings}
\date{}

\maketitle
\begin{center}
Madison Gamble, Spencer Hamblen, Blake Schildhauer, and Chung Truong

Department of Mathematics and Computer Science

McDaniel College

Westminster, MD  21157

USA

mlg011@connections.mcdaniel.edu

shamblen@mcdaniel.edu

blakeschildhauer@gmail.com

hanhchung.truong@stella.edu.vn
\end{center}



\begin{abstract}
We investigate a version of Waring's Problem over quaternion rings, focusing on cubes in quaternion rings with integer coefficients.  We determine the global upper and lower bounds for the number of cubes necessary to represent all such quaternions.
\end{abstract}

\section{Introduction and Definitions}

\begin{theorem}[Waring's Problem/Hilbert-Waring Theorem]
For every integer $k \geq 2$ there exists a positive integer $g(k)$ such that every positive integer is the sum of at most $g(k)$ $k$-th powers of integers.
\end{theorem}

The idea behind Waring's Problem -- examining sums of powers -- can be easily extended to any ring.  (For example, number fields \cite{siegel} and polynomial rings over finite fields \cite{newman}.)  For an excellent and thorough exposition of the research on Waring's Problem and its generalizations, see Vaughan and Wooley \cite{wooley}.  We will specifically look at sums of cubes in quaternion rings, extending the previous work on sum of squares begun in Cooke, Hamblen, and Whitfield \cite{chw}.

\begin{definition} Let $LQ_{a,b}$ denote the quaternion ring
\[ \{\alpha_0 + \alpha_1 {\bf i} + \alpha_2 {\bf j} + \alpha_3 {\bf k} \mid \alpha_n,a,b \in {\mathbb Z}, {\bf i}^2 = -a, {\bf j}^2 = -b, {\bf i}{\bf j}=-{\bf j}{\bf i}={\bf k}\}.\]
Let $LQ_{a,b}^n$ denote the additive group generated by all $n$th powers in $LQ_{a,b}$.
\end{definition}

Note here that ${\bf k}^2 = -ab$, and that if $a = b = 1$, we have the {\em Lipschitz quaternions}.  We then have the following analogue of Waring's Problem.

\begin{conjecture}
For every integer $k \geq 2$ and all positive integers $a,b$ there exists a positive integer $g_{a,b}(k)$ such that every element of $LQ_{a,b}^k$ can be written as the sum of at most $g_{a,b}(k)$ $k$-th powers of elements of $LQ_{a,b}$.
\end{conjecture}

In contrast with the case when $k=2$, it is much harder when an element of a ring can be represented as a sum of a small number of cubes.  For example, it was only recently determined \cite{booker} that 33 is the sum of 3 integer cubes.  Our goal in this paper, therefore, is to determine global upper and lower bounds for $g_{a,b}(3)$, the number of cubes necessary to represent all elements of $LQ_{a,b}^3$.  We have the following main result.

\begin{theorem} \label{mainthm}
Let $a,b$ be positive integers. Then
\begin{itemize}
	\item if $3 \nmid a$ or $3 \nmid b$, then $3 \leq g_{a,b}(3) \leq 6$, and
	\item if $3 \mid a$ and $3 \mid b$, then $4 \leq g_{a,b}(3) \leq 5$.
\end{itemize}
\end{theorem}

The upper bounds of Theorem \ref{mainthm} are given in Section 2, following an algorithmic approach based on cubic algebraic identities.  The lower bounds are given in Section 3.

It seems quite possible that the lower bounds in Theorem \ref{mainthm} are the actual values for $g_{a,b}(3)$.  A number of individual quaternions were tested in SAGE, and all were found to be expressible as the minimum number of cubes.  Additionally, the identities of Equations (\ref{cube1}) and (\ref{cube2}), while very useful for our upper bound proof, are by no mean optimal.  A search for similar identities involving quaternions was unsuccessful, due to the complications introduced by non-commutativity.

Lastly, it should be noted that Propositions \ref{n3abup} and \ref{3abup} were both initially proven by checking individual residue classes in SAGE.  While we were able to cover all possible cases, more theoretical versions of the proofs are provided here.

\section{$LQ_{a,b}^3$ and Upper Bounds}

Recall that $LQ_{a,b}^3$ is the additive subgroup generated by all cubes in $LQ_{a,b}$.  Our first goal is to determine the shape of elements in $LQ_{a,b}^3$; we therefore first give the general forms of cubes in $LQ_{a,b}$.  If $\alpha = \alpha_0 + \alpha_1{\bf i}+\alpha_2{\bf j}+\alpha_3{\bf k}$, we have
\begin{align}
	\alpha^3 & = \alpha_0^3 - 3a\alpha_0\alpha_1^2 - 3b\alpha_0\alpha_2^2 - 3ab\alpha_0\alpha_3^2 \label{cubeeq}  \\
    & \quad + (3\alpha_0^2\alpha_1 - a\alpha_1^3 - b\alpha_1\alpha_2^2 - ab\alpha_1\alpha_3^2) {\bf i} \notag \\
    & \quad + (3\alpha_0^2\alpha_2 - a\alpha_1^2\alpha_2 - b\alpha_2^3 - ab\alpha_3\alpha_3^2) {\bf j} \notag \\
    & \quad + (3\alpha_0^2\alpha_3 - a\alpha_1^2\alpha_3 - b\alpha_1\alpha_2^2 - ab\alpha_3^3) {\bf k} \notag
\end{align}

We can simplify this equation by noting common factors in each of the coefficients on the right side of Equation (\ref{cubeeq}).  For $\alpha = \alpha_0 + \alpha_1{\bf i} + \alpha_2{\bf j} + \alpha_3{\bf k}$, let 
\begin{equation} \label{Pal}
P_{\alpha} = a\alpha_1^2 + b\alpha_2^2 + ab\alpha_3^2.
\end{equation}
We then have 
\begin{equation} \label{cubeeq2} 
	\alpha^3 = (\alpha_0^2 - 3 P_{\alpha}) \alpha_0  + (3\alpha_0^2 - P_{\alpha} ) \left(\alpha_1 {\bf i} + \alpha_2 {\bf j}  + \alpha_3 {\bf k} \right)
\end{equation}

Additionally, we will make frequent use of the following two identities:

\begin{align}
	6z  &=(z+1)^3+(z-1)^3+(-z)^3+(-z)^3 \label{cube1}\\
    6z+3&=(-z-5)^3+(z+1)^3+(-2z-6)^3+(2z+7)^3 \label{cube2}
\end{align}

These two identities, and these proofs, are inspired by Cohn's results \cite{cohn1, cohn2} on sums of cubes in quadratics fields: $g_{{\mathbb Z}[i]}(3) = 4$ and $g_{{\mathbb Z}[\sqrt{d}]}(3) \leq 5$.

We start by treating the case when $3 \nmid a$ or $3 \nmid b$.

\begin{prop} \label{n3ab3}
If $3 \nmid a$ or $3 \nmid b$, then $LQ_{a,b}^3 = LQ_{a,b}$.
\end{prop}

Note that in the Lipschitz quaternions ($a=b=1$), this follows from Theorem 1.1 of \cite{pollack}.

\begin{prop} \label{n3abup}
If $3 \nmid a$ or $3 \nmid b$, then every element of $LQ_{a,b}^3$ can be written as the sum of at most 6 cubes of elements in $LQ_{a,b}$.
\end{prop}

We will prove that every element of $LQ_{a,b}$ can be written as the sum of at most 6 cubes, which yields both propositions.

\begin{proof}
First, note that by Equations (\ref{cube1}) and (\ref{cube2}), we immediately have that every element in $LQ_{a,b}$ that is a multiple of 6, or 3 more than a multiple of 6, can be written as the sum of 4 cubes.  It then suffices to restrict our attention to the resulting residue classes, and we need only consider the residue of $a,b$ mod 6.  We will break the problem into two cases, and in each case will need two supporting Lemmas.  

Our two cases are as follows:
\begin{itemize}
	\item \underline{Case 1:} Suppose $3 \nmid ab$, and at least one of $a$ or $b$ is congruent to $2 \bmod 3$, and
	\item \underline{Case 2:} All other cases: either $a \equiv b \equiv 1 \bmod 3$, or exactly one of $a$ and $b$ is divisible by 3.
\end{itemize}

For the following Lemmas, we let $\text{Re}(x)$ be the real part of $x$ and $\text{Im}(x)$ be the imaginary or pure part of $x$.  That is, if $x= x_0 + x_1 {\bf i} + x_2 {\bf j} + x_3 {\bf k}$, then $\text{Re}(x) = x_0$ and $\text{Im}(x) = x_1 {\bf i} + x_2 {\bf j} + x_3 {\bf k}$.  Additionally, we write $\text{Im}(x) \equiv \text{Im}(y) \bmod 6$ if 6 divides each of the coefficients of $\text{Im}(x-y)$.  Lastly, for $n \in {\mathbb Z}$, we write $\overline{n}$ for the least non-negative residue of $n \bmod 6$; that is $\overline{n} \equiv n \bmod 6$ and $\overline{n} \in \{0,1, \dots, 5\}$.

\begin{lemma}\label{cubeset1}
Suppose we are in Case 1: $3 \nmid ab$, and at least one of $a$ or $b$ is congruent to $2 \bmod 3$, and let 
\[S = \{ \alpha \in LQ_{a,b} \mid 2 \nmid \alpha_0 \text{ and } 3 \nmid \alpha_1\alpha_2\alpha_3\}.\]
Then, for all $\alpha \in S$, there exists $x \in LQ_{a,b}$ such that $\text{Re}(x^3) \equiv \text{Re}(\alpha) \bmod 3$ and $\text{Im}(x^3) \equiv \text{Im}(\alpha) \bmod 6$.  
\end{lemma}

Note that as an immediately corollary of Lemma \ref{cubeset1} and Equations (\ref{cube1}) and (\ref{cube2}), every element of $S$ can be written as the sum of at most 5 cubes.

\begin{proof}
Take $\alpha = \alpha_0 + \alpha_1 {\bf i} + \alpha_2 {\bf j} + \alpha_3 {\bf k} \in S$.  Then let $x =  x_0 + x_1 {\bf i} + x_2 {\bf j} + x_3 {\bf k}$, where $x_{\ell} = \overline{\alpha_{\ell}}$ for $\ell \in \{1,2,3\}$ and $x_0 = \overline{\alpha_0} - 3\delta_{\alpha}$, where
\[\delta_{\alpha} = \begin{cases}
	1, & \text{if } P_{\alpha} \text{ is odd;} \\
	0, & \text{otherwise.}
\end{cases}	\]
By Equation (\ref{cubeeq2}), it suffices to show that $x_0^3 - 3x_0P_x \equiv \alpha_0 \bmod 3$, and $x_{\ell}(3x_0^2 - P_x) \equiv \alpha_{\ell} \bmod 6$ for $\ell \in \{1,2,3\}$.

We then have
\begin{equation} \label{realcube}
	x_0^3 - 3x_0P_x  = (\overline{\alpha_0} - 3\delta_{\alpha})^3 - 3(\overline{\alpha_0} - 3\delta_{\alpha})P_x  \equiv \alpha_0^3 \equiv \alpha_0 \bmod 3,
\end{equation}
so $\text{Re}(x^3) \equiv \text{Re}(\alpha) \bmod 3$.  Then, note that in this case we have $\alpha \in S$, $\alpha_1^2 \equiv \alpha_2^2 \equiv \alpha_3^2 \equiv 1 \bmod 3$, so 
\begin{align*}
	P_{\alpha} & \equiv a\cdot 1 + b \cdot 1 + ab \cdot 1 \bmod 3 \\
	& \equiv (a + 1)(b+1) - 1\bmod 3
\end{align*}
Since at least one of $a$ or $b$ is congruent to $2 \bmod 3$, we must have that $P_{\alpha} \equiv 2 \bmod 3$.  Therefore if $\delta_{\alpha} = 1$, then $P_{\alpha} \equiv 5 \bmod 6$, and if $\delta_{\alpha} = 0$, then $P_{\alpha} \equiv 2 \bmod 6$; in either case, $3\delta_{\alpha} - P_{\alpha} \equiv -2 \bmod 6$.

Then note that since $P_x \equiv P_{\alpha} \bmod 6$ (since by definition $\text{Im}(x) \equiv \text{Im}(\alpha) \bmod 6$) and $\alpha_0$ is odd, we have
\begin{align*}
	3x_0^2 - P_x = 3(\overline{\alpha_0} - 3\delta_{\alpha})^2 - P_x & \equiv 3\alpha_0^2 + 3\delta_\alpha - P_{\alpha} \bmod 6\\
	& \equiv 3 - 2 = 1\bmod 6
\end{align*}
Therefore $x_{\ell}(3x_0^2 - P_x) \equiv \alpha_{\ell} \bmod 6$ for $\ell \in \{1,2,3\}$, so $\text{Im}(x^3) \equiv \text{Im}(\alpha) \bmod 6$, which completes the proof.
\end{proof}

\begin{lemma}\label{sumset1}
Suppose $3 \nmid ab$, and at least one of $a$ or $b$ is congruent to $2 \bmod 3$, and let $S$ be defined as in Lemma \ref{cubeset1}.  Then, for all $\alpha \in LQ_{a,b}$, there exists $\alpha',\alpha'' \in S$ such that $\text{Re}(\alpha' + \alpha'') \equiv \text{Re}(\alpha) \bmod 3$ and $\text{Im}(\alpha' + \alpha'') \equiv \text{Im}(\alpha) \bmod 6$.
\end{lemma}

\begin{proof}
Notice that elements of $S$ can have real coefficient equivalent to 1, 3, or $5 \bmod 6$, and can have imaginary coefficients equivalent to 1, 2, 4, or $5 \bmod 6$.  The first conclusion then follows since the real coefficients cover all residue classes$\mod 3$, and the second follows from the fact that in ${\mathbb Z}_6$, $\{1,2,4,5\} + \{1,2,4,5\} = {\mathbb Z}_6$.
\end{proof}

As a consequence of Lemmas \ref{cubeset1} and \ref{sumset1}, for all $\alpha \in LQ_{a,b}$, there exists $x_1, x_2 \in LQ_{a,b}$ such that $\alpha - x_1^3 + x_2^3$ is either a multiple of 6, or 3 more than a multiple of 6; Equations (\ref{cube1}) and (\ref{cube2}) then imply that under the hypotheses of Case 1, every element of $LQ_{a,b}$ can be written as the sum of at most 6 cubes.  We have therefore proven Propositions \ref{n3ab3} and \ref{n3abup} in the case when $3 \nmid ab$, and at least one of $a$ or $b$ is congruent to $2 \bmod 3$.

We then move to Case 2, where we suppose that we are in one of the following cases:
\begin{itemize}
	\item \underline{Case 2a}: $a \equiv b \equiv 1 \bmod 3$.
	\item \underline{Case 2b}: Exactly one of $a$ and $b$ is divisible by 3, and the other is $2 \bmod 3$.  Without loss of generality, in this case we assume $a \equiv 2 \bmod 3$ and $b \equiv 0 \bmod 3$.
	\item \underline{Case 2c}: Exactly one of $a$ and $b$ is divisible by 3, and the other is $1 \bmod 3$.  Without loss of generality, in this case we assume $a \equiv 1 \bmod 3$ and $b \equiv 0 \bmod 3$.
\end{itemize}	
	
\begin{lemma}\label{cubeset2}
Given $a$ and $b$ satisfying one of the cases above, let 
\[T_2 = \{ \alpha \in LQ_{a,b} \mid 2 \nmid \alpha_0 \text{ and } 3 \nmid \alpha_1\alpha_3 \text{ and } 3 \mid \alpha_2\},\]
\[T_3 = \{ \alpha \in LQ_{a,b} \mid 2 \nmid \alpha_0 \text{ and } 3 \nmid \alpha_1\alpha_2 \text{ and } 3 \mid \alpha_3\},\]
and $T = T_2 \cup T_3$.  Then, for all $\alpha \in T$, there exists $x \in LQ_{a,b}$ such that $\text{Re}(x^3) \equiv \text{Re}(\alpha) \bmod 3$ and $\text{Im}(x^3) \equiv \text{Im}(\alpha) \bmod 6$.  
\end{lemma}

\begin{proof} The proofs in each subcase are very similar to that of Lemma \ref{cubeset1}; we will only highlight where the definitions and calculations differ.

Take $\alpha = \alpha_0 + \alpha_1 {\bf i} + \alpha_2 {\bf j} + \alpha_3 {\bf k} \in S$, let $x_0 = \overline{\alpha_0} - 3\delta_{\alpha}$ as defined in Lemma \ref{cubeset1}, and let 
\[x_{\ell} = \begin{cases}
	\overline{\alpha_{\ell}}, & \text{in Cases 2a and 2b; }\\
	6 - \overline{\alpha_{\ell}} ,& \text{in Case 2c}.
\end{cases}	\]
Immediately by Equation (\ref{realcube}) in Lemma \ref{cubeset1}, we have that $\text{Re}(x^3) \equiv \text{Re}(\alpha) = \alpha_0 \bmod 3$.

Then, for $\alpha \in T_2$, we have $\alpha_1^2 \equiv \alpha_3^2 \equiv 1 \bmod 3$ and $\alpha_2^2 \equiv 0 \bmod 3$, so from Equation (\ref{Pal}):
\[
	P_{\alpha} \equiv 
	\begin{cases}
		 2 \quad \equiv 1\cdot 1 + 1 \cdot 0 + 1 \cdot 1\bmod 3, & \text{in Case 2a;} \\
		 2 \quad \equiv  2\cdot 1 + 0 \cdot 0 + 0 \cdot 1  \bmod 3, & \text{in Case 2b;} \\
		 1 \quad \equiv 1\cdot 1 + 0 \cdot 0 + 0 \cdot 1  \bmod 3, & \text{in Case 2c}. 
	\end{cases}	 
\]
Note that in all of these Cases, $b \equiv ab \bmod 3$, so for $\alpha \in T_3$, the values of $P_{\alpha}$ mod 3 are the same as for $\alpha \in T_2$.

Therefore, in Cases 2a and 2b, if $\delta_{\alpha} = 1$, then $P_{\alpha} \equiv 5 \bmod 6$, and if $\delta_{\alpha} = 0$, then $P_{\alpha} \equiv 2 \bmod 6$; either way, $3\delta_{\alpha} - P_{\alpha} \equiv -2 \bmod 6$.  Since $P_x \equiv P_{\alpha} \bmod 6$ and $\alpha_0$ is odd, we have
\begin{align*}
	3x_0^2 - P_x = 3(\overline{\alpha_0} - 3\delta_{\alpha})^2 - P_x & \equiv 3\alpha_0^2 + 3\delta_\alpha - P_{\alpha} \bmod 6\\
	& \equiv 3 - 2 = 1\bmod 6
\end{align*}
Therefore $\text{Im}(x^3) \equiv \text{Im}(\alpha) \bmod 6$, which completes the proof for Cases 2a and 2b.

In Case 2c, if $\delta_{\alpha} = 1$, then $P_{\alpha} \equiv 1 \bmod 6$, and if $\delta_{\alpha} = 0$, then $P_{\alpha} \equiv 4 \bmod 6$; either way, $3\delta_{\alpha} - P_{\alpha} \equiv 2 \bmod 6$.  The same calculation as above then yields
\[ 3x_0^2 - P_x \equiv 3 + 2 \equiv -1\bmod 6\]
But, as we have defined $x_{\ell} = 6 - \overline{\alpha_{\ell}}$ in this case, we have
\[x_{\ell} (3x_0^2 - P_x) \equiv (6 - \overline{\alpha_{\ell}})(-1) \equiv \alpha_{\ell} \bmod 6\]
for $\ell \in \{1, 2, 3\}$, which implies $\text{Im}(x^3) \equiv \text{Im}(\alpha) \bmod 6$, completing the proof for Case 2c.
\end{proof}

\begin{lemma}\label{sumset2}
Given $a$ and $b$ satisfying Case 2, let $T$ be defined as in Lemma \ref{cubeset2}. Then, for all $\alpha \in LQ_{a,b}$, there exists $\alpha',\alpha'' \in T$ such that $\text{Re}(\alpha' + \alpha'') \equiv \text{Re}(\alpha) \bmod 3$ and $\text{Im}(\alpha' + \alpha'') \equiv \text{Im}(\alpha) \bmod 6$.
\end{lemma}

\begin{proof} In light of Lemma \ref{sumset1}, if $3 \mid \alpha_2$ or $3 \mid \alpha_3$, we can choose $\alpha'$ and $\alpha''$ both to be in $T_2$ or $T_3$, respectively.  If $3 \nmid \alpha_2\alpha_3$, then there exists $\alpha' \in T_2$ and $\alpha'' \in T_3$ satisfying the conclusions.
\end{proof}

This completes the proofs of Propositions \ref{n3ab3} and \ref{n3abup}: as in Case 1, Lemmas \ref{cubeset2} and \ref{sumset2} imply that in Case 2, every element of $LQ_{a,b}$ can be written as the sum of at most 6 cubes.
\end{proof}

If $3 \mid a$ and $3 \mid b$, there is slightly more work to do, as not all elements of the ring can be written as the sum of cubes.

\begin{prop}
If $3 \mid a$ and $3 \mid b$, then 
\[LQ_{a,b}^3=\{ \alpha_0 + 3\alpha_1 {\bf i} + 3\alpha_2 {\bf j} + 3\alpha_3 {\bf k} \mid {\bf i}^2 = -a, {\bf j}^2 = -b, 
={\bf i}{\bf j} = -{\bf j} {\bf i} = {\bf k}, \alpha_n \in {\mathbb Z} \}.\]
\end{prop}

\begin{proof}
Note that if $3 \mid a$ and $3 \mid b$, then for all $\alpha \in LQ_{a,b}$, we have $3 \mid P_{\alpha}$ from Equation \ref{Pal}.  Then by Equation \ref{cubeeq2}, we have that each of the imaginary coefficients (the coefficients of ${\bf i}, {\bf j}, {\bf k}$) are each divisible by 3, showing that the form above is necessary for all elements of $LQ_{a,b}^3$.

The sufficiency of the above form is then the result of the proof of Proposition \ref{3abup}, which shows that every element of this form can be written as the sum of at most 5 cubes.
\end{proof}

\begin{prop} \label{3abup}
If $3 \mid a$ and $3 \mid b$, then every element of $LQ_{a,b}^3$ can be written as the sum of at most 5 cubes of elements in $LQ_{a,b}$.
\end{prop}

\begin{proof}
In light of Equations \ref{cube1} and \ref{cube2}, it suffices to show that for all elements $\alpha \in LQ_{a,b}^3$, there exists $x \in LQ_{a,b}$ such that $\text{Re}(x^3) \equiv \text{Re}(\alpha) \bmod 3$ and $\text{Im}(x^3) \equiv \text{Im}(\alpha) \bmod 6$.  

Take $\alpha = \alpha_0 + \alpha_1 {\bf i} + \alpha_2 {\bf j} + \alpha_3 {\bf k} \in LQ_{a,b}^3$.  Then let $x_{\ell} = \overline{\alpha_{\ell}}$ for $\ell \in \{1,2,3\}$ and $x_0 = \overline{\alpha_0} - 3\delta_{\alpha}$, where
\[\delta_{\alpha} = \begin{cases}
	1, & \text{if } P_{\alpha} \equiv \alpha_0 \bmod 2; \\
	0, & \text{otherwise.}
\end{cases}	\]
We immediately get $\text{Re}(x^3) \equiv \text{Re}(\alpha) \bmod 3$ by the calculations in Lemma \ref{cubeset1}.

For $\alpha \in LQ_{a,b}^3$, since $3 \mid a$ and $3 \mid b$, we have $P_{\alpha} \equiv 0 \bmod 3$.  Therefore if $\delta_{\alpha} = 1$, then $\alpha_0$ is odd and $P_{\alpha} \equiv 3 \bmod 6$, or $\alpha_0$ is even and $P_{\alpha} \equiv 0 \bmod 6$.  If $\delta_{\alpha} = 0$, then $\alpha_0$ is odd and $P_{\alpha} \equiv 0 \bmod 6$, or $\alpha_0$ is even and $P_{\alpha} \equiv 3 \bmod 6$.  Specifically, an {\em odd} number of $\alpha_0$, $\delta_{\alpha}$, and $P_{\alpha}$ will be odd.  We then have
\begin{align*}
	3x_0^2 - P_x = 3(\overline{\alpha_0} - 3\delta_{\alpha})^2 - P_x & \equiv 3\alpha_0^2 + 3\delta_\alpha - P_{\alpha} \bmod 6 \\
	& \equiv 3 \bmod 6
\end{align*}

Then, since $\alpha \in LQ_{a,b}^3$, $\alpha_{\ell}$ is a multiple of 3 for $\ell \in \{1,2,3\}$, so $3\alpha_{\ell} \equiv \alpha_{\ell} \bmod 6$.  But these are now exactly the mod 6 imaginary coefficients of $x^3$.

Therefore $\text{Im}(x^3) \equiv \text{Im}(\alpha) \bmod 6$, which completes the proof.
\end{proof}

\section{Lower Bounds}

We now prove the lower bounds of Theorem \ref{mainthm} via example.  

\begin{prop} \label{n3lb}
If $3 \nmid a$ or $3 \nmid b$, then $3+3{\bf i}$ cannot be written as the sum of 2 cubes in $LQ_{a,b}$.
\end{prop}

\begin{proof}
Suppose $x, y \in LQ_{a,b}$ are such that 
\begin{equation}
    3+ 3{\bf i} = x^3 + y^3, \label{not3eq}
\end{equation}
and write $x = x_0 + x_1{\bf i} + x_2{\bf j} + x_3{\bf k}$, $y = y_0 + y_1{\bf i} + y_2{\bf j} + y_3{\bf k}$ with $x_n, y_n \in {\mathbb Z}$.  We then have the following four equations from the coefficients of Equation (\ref{not3eq}):
\begin{align}
    x_0^3-3x_0P_x+y_0^3-3y_0P_y &=3 &(\text{real coefficient}) \label{re1} \\
    3x_0^2x_1-x_1P_x+3y_0^2y_1-y_1P_y &=3 & ({\bf i}  \text{ coefficient}) \label{ico1}\\
    3x_0^2x_2-x_2P_x +3y_0^2y_2-y_2P_y &=0 & ({\bf j}  \text{ coefficient}) \label{jco1}\\
    3x_0^2x_3-x_3P_x+3y_0^2y_3-y_3P_y &=0 & ({\bf k}  \text{ coefficient}) \label{kco1}
\end{align}

From Equation (\ref{re1}), we get $x_0^3 + y_0^3 \equiv 0 \bmod 3$; as the only cubes mod 9 are 0, 1, and 8, we immediately get $x_0^3 + y_0^3 \equiv 0 \bmod 9$.  Since $x_0^3 \equiv x_0 \bmod 3$, we also  get
\begin{equation} \label{xminy}
    x_0 + y_0 \equiv 0 \bmod 3.
\end{equation}

We can then examine Equation (\ref{re1}) mod 9 and simplify:
\begin{align}
    x_0^3-3x_0P_x+y_0^3-3y_0P_y & \equiv 3 \bmod 9 \notag\\
    -3x_0P_x-3y_0P_y & \equiv 3 \bmod 9 \notag\\    
    -x_0P_x-y_0P_y & \equiv 1 \bmod 3 \notag\\    
    -x_0P_x-y_0P_y & \equiv 1 \bmod 3 \notag\\    
    y_0(P_x-P_y) & \equiv 1 \bmod 3 \label{n31}   
\end{align}
If we first assume (without loss of generality) that $P_x \equiv 0 \bmod 3$.  Then $P_y \not\equiv 0 \bmod 3$, and Equations (\ref{ico1}), (\ref{jco1}), (\ref{kco1}) become
\begin{align*}
    -y_1P_y & \equiv 0 \bmod 3 \\
    -y_2P_y & \equiv 0 \bmod 3 \\
    -y_3P_y & \equiv 0 \bmod 3
\end{align*}
Therefore $y_1 \equiv y_2 \equiv y_3 \equiv 0 \bmod 3$, which implies that $P_y \equiv 0 \bmod 3$, a contradiction.  Therefore $P_x, P_y \not\equiv 0 \bmod 3$.

We additionally have from Equation (\ref{n31}) that $P_x \not\equiv P_y \bmod 3$, so assume $P_x \equiv 1 \bmod 3$ and $P_y \equiv 2 \bmod 3$.  From Equations (\ref{ico1}), (\ref{jco1}), and (\ref{kco1}) we have $x_n \equiv 2y_n \bmod 3$ for $n \in \{1,2,3\}$, which implies $x_n^2 \equiv y_n^2 \bmod 3$.  We then have 
\begin{align*}
    1 \equiv P_y - P_x & \equiv (ay_1^2 + by_2^2 + aby_3^2) - (ax_1^2 + bx_2^2 + abx_3^2) \bmod 3\\
    & \equiv a(y_1^2 - x_1^2) + b(y_1^2 - x_1^2) + ab(y_3^2 - x_3^2) \bmod 3 \\
    & \equiv 0 \bmod 3
\end{align*}
We therefore have the contradiction in this case, which completes the proof.
\end{proof}

\begin{prop} \label{3lb}
If $3 \mid a$ and $3 \mid b$, then $4$ cannot be written as the sum of 3 cubes in $LQ_{a,b}$.
\end{prop}

\begin{proof}
Suppose $x, y, z \in LQ_{a,b}$ are such that $4= x^3 + y^3 + z^3$.  Examining the real coefficients of Equation (\ref{not3eq}), we get the following (similar to Equation (\ref{re1})):
\begin{equation} \label{re2}
   x_0^3-3x_0P_x+y_0^3-3y_0P_y + z_0^3-3z_0P_z = 4
\end{equation}

Note that since $3\mid a$ and $3\mid b$, we have $P_x \equiv P_y \equiv P_z \equiv 0 \bmod 3$; therefore Equation (\ref{re2}) becomes
\[    x_0^3+y_0^3 + z_0^3 \equiv 4 \bmod 9, \]
which has no integer solutions.
\end{proof}

Propositions \ref{n3abup}, \ref{3abup}, \ref{n3lb}, and \ref{3lb} then complete the proof of Theorem \ref{mainthm}.

\section{Acknowledgments}

The authors would like to thank the McDaniel College Student-Faculty Summer Research Fund and Research and Creativity Fund for supporting their research.

\bibliographystyle{plainnat}

\end{document}